\documentclass[10pt,oneside,reqno]{amsart} 
\usepackage{amsmath,amssymb, amsfonts}
\usepackage[dvipdfmx]{graphicx}
\usepackage{float}
\usepackage{multicol}
\usepackage{ascmac}
\usepackage{mathtools}
\usepackage{bm}
\usepackage{amscd}
\usepackage{comment}
\usepackage{mathrsfs}
\usepackage[all]{xy}
\usepackage{color}
\usepackage{shuffle}
\usepackage{amsthm}
\usepackage{setspace}
\usepackage{cases}
\usepackage{paralist}

\title[CSMZV spans whole spaces CMZV]{Cyclotomic symmetric multiple zeta values span the space of cyclotomic multiple zeta values}
\author{Takumi Anzawa}
\address{graduate school of mathematics, Nagoya University, furo-cho, chikusa-ku, Nagoya, 464-8602, japan}
\email{m20001s@math.nagoya-u.ac.jp}
\date{\today} 
\usepackage[margin=30mm]{geometry}
\theoremstyle{theorem}{
\newtheorem{df}{Definition}[section]}
\theoremstyle{theorem}{
\newtheorem{thm}[df]{Theorem}}

\newtheorem{lem}[df]{Lemma}
\newtheorem{cor}[df]{Corollary}
{\theoremstyle{remark}
\newtheorem{rmk}[df]{Remark}}

\DeclareMathOperator{\dep}{dp}

\DeclareMathOperator{\wt}{wt}

\DeclareMathOperator{\Span}{span}

\newcommand{\ide}[1]{\mathfrak{#1}}

\newcommand{\A}{\mathscr{A}}

\newcommand{\engMonth}{
	\ifcase\month
		\or January 
		\or February 
		\or March 
		\or April
		\or May 
		\or June
		\or July 
		\or August 
		\or September
		\or October 
		\or November 
		\or December\fi
} 
\newcommand{\fullday}{
	\engMonth \, \the\day, \the\year	
} 

\let\S\relax
\newcommand{\S}{\mathcal{S}}

\newcommand{\RS}{\mathcal{RS}}

\newcommand{\cmzv}[4]{\zeta\left(\begin{matrix} {#1} \hspace{-0.3cm}&,\ldots,&\hspace{-0.3cm} {#2} \\ {#3} \hspace{-0.3cm}&,\ldots,&\hspace{-0.3cm}{#4}\end{matrix} \right)}

\newcommand{\csmzv}[4]{\zeta^{\shuffle}\left(\begin{matrix} {#1} \hspace{-0.3cm}&,\ldots,&\hspace{-0.3cm} {#2} \\ {#3} \hspace{-0.3cm}&,\ldots,&\hspace{-0.3cm}{#4}\end{matrix} \right)}

\newcommand{\chmzv}[4]{\zeta^{*}\left(\begin{matrix} {#1} \hspace{-0.3cm}&,\ldots,&\hspace{-0.3cm} {#2} \\ {#3} \hspace{-0.3cm}&,\ldots,&\hspace{-0.3cm}{#4}\end{matrix} \right)}

\newcommand{\cssmzv}[4]{\zeta^{\shuffle}_{\S(N,\alpha)}\left(\begin{matrix} {#1} \hspace{-0.3cm}&,\ldots,&\hspace{-0.3cm} {#2} \\ {#3} \hspace{-0.3cm}&,\ldots,&\hspace{-0.3cm}{#4}\end{matrix} \right)}

\newcommand{\chsmzv}[4]{\zeta^{*}_{\S(N,\alpha)} \left(\begin{matrix} {#1} \hspace{-0.3cm}&,\ldots,&\hspace{-0.3cm} {#2} \\ {#3} \hspace{-0.3cm}&,\ldots,&\hspace{-0.3cm}{#4}\end{matrix} \right)}

\newcommand{\cSmzv}[4]{\zeta_{\S(N,\alpha)}\left(\begin{matrix} {#1} \hspace{-0.3cm}&,\ldots,&\hspace{-0.3cm} {#2} \\ {#3} \hspace{-0.3cm}&,\ldots,&\hspace{-0.3cm}{#4}\end{matrix} \right)}

\newcommand{\crsmzv}[4]{\zeta_{\RS(N,\alpha)} \left(\begin{matrix} {#1} \hspace{-0.3cm}&,\ldots,&\hspace{-0.3cm} {#2} \\ {#3} \hspace{-0.3cm}&,\ldots,&\hspace{-0.3cm}{#4}\end{matrix} \right)}

\numberwithin{equation}{section}

\allowdisplaybreaks[4]
\begin{document}
\setstretch{0.9}
\maketitle
\begin{abstract}
In this paper, we show that the cyclotomic symmetric multiple zeta values, independently proposed by Jarossay, Singar and Zhao, and Tasaka, span the space of the cyclotomic multiple zeta values modulo $\pi i$.
\end{abstract}
\tableofcontents

\section{Introduction}

Throughout this paper,
let $N$ be a positive integer, $\alpha \in \mathbb{Z}/N\mathbb{Z}$, $\Gamma_N\subset \mathbb{C}$ be the set of $N$-th roots of unity, and $K_N$ be the $N$-th cyclotomic field over $\mathbb{Q}$.

For  $\mathbf{k}:= (k_1,\ldots,k_r) \in \mathbb{Z}_{>0}^r$ and $\boldsymbol{\xi} := (\xi_1,\ldots,\xi_r) \in \Gamma_N^r$ with $(k_r,\xi_r)\neq (1,1)$, 
we define the \textbf{cyclotomic multiple zeta value} (CMZV) by
\[
\cmzv{\xi_1}{\xi_r}{k_1}{k_r} := \sum_{0<m_1<\cdots <m_r}\frac{\xi_1^{m_1}\cdots \xi_r^{m_r}}{m_1^{k_1}\cdots m_r^{k_r}}.
\] 
The condition $(k_r,\xi_r) \neq (1,1)$ ensures the convergence of this series. We call such a pair $(\mathbf{k}, \boldsymbol{\xi}) \in \mathbb{Z}_{>0}^r\times \Gamma_N^r$ with $(k_r,\xi_r)\neq (1,1)$ an admissible index and we set $\zeta\left( \begin{matrix} \emptyset \\ \emptyset \end{matrix}\right) = 1$.
We define the weight of $\mathbf{k}$ by $\wt \mathbf{k}:=k_1+\cdots + k_r$
and we define the depth of $\mathbf{k} \in \mathbb{Z}_{>0}^r$ and $\boldsymbol{\xi}\in \Gamma_N^r$ by $\dep \mathbf{k}:= \dep \boldsymbol{\xi} := r$.
By $\mathcal{Z}^N$, we denote the $K_N$-algebra generated by all CMZVs. Put $\widetilde{\mathcal{Z}}^N := \mathcal{Z}^N [\pi i] \subset \mathbb{C}$. 

For $ (k_1,\ldots,k_r) \in \mathbb{Z}_{>0}^r$ and $(\xi_1,\ldots,\xi_r) \in \Gamma_N^r$,
we define \textbf{cyclotomic symmetric multiple zeta values} (CSMZVs), which were introduced independently by Jarossay \cite{Jarossay}, Singer and Zhao \cite{Singer-Zhao}, and Tasaka \cite{Tasaka}\footnote{Tasaka calls a lift of this symmetric multiple zeta value a cyclotomic symmetric multiple zeta value.} as an element of $\widetilde{\mathcal{Z}}^N /(\pi i\widetilde{\mathcal{Z}}^N )$, by
\begin{align*}
\begin{split}
\cSmzv{\xi_1}{\xi_r}{k_1}{k_r} :=&  \sum_{j = 0}^r  (\xi_{i+1}\cdots \xi_r)^{\alpha} (-1)^{k_{j+1}+\cdots + k_r}\\
&\times\left(\csmzv{\xi_1}{\xi_j}{k_1}{k_j}\csmzv{\overline{\xi_r}}{\overline{\xi_{j+1}}}{k_r}{k_{j+1}} \mod \pi i \right).
\end{split}
\end{align*} 
Here, $\zeta^{\shuffle}$ is the $\shuffle$-regularized CMZV (see Section \ref{sec:CMZV} for the definition).
Let us define $\mathcal{Z}^{\S(N,\alpha)}$ as the $K_N$-linear subspace of $\widetilde{\mathcal{Z}}^N/(\pi i \widetilde{\mathcal{Z}}^N)$ spanned by all CSMZVs.

The study of CSMZVs is motivated by the cyclotomic analogue of Kaneko-Zagier conjecture ($N = 1$ in \cite[Primitive Conjecture]{Kaneko}, general $N$ in \cite[Conjecture 6.1]{Tasaka}), which predicts, 
for $N\in\mathbb{Z}_{>0}$ and $\alpha \in (\mathbb{Z}/N\mathbb{Z})^{\times}$, the existence of a well-defined $K_N$-algebraic isomorphism
\begin{align}
\begin{array}{ccc}
\mathcal{Z}^{\mathscr{A}(N,\alpha)} &\longrightarrow &  \widetilde{\mathcal{Z}}^N/(\pi i \widetilde{\mathcal{Z}}^N)  \\
\rotatebox{90}{$\in$}&& \rotatebox{90}{$\in$}\\
\zeta_{\mathscr{A}(N,\alpha)}(\mathbf{k})  & \longmapsto & \zeta_{\S(N,\alpha)}(\mathbf{k}).
\end{array}
\label{eq:KZ-conj}
\end{align}
Here, $\zeta_{\mathscr{A}(N,\alpha)}$ is a cyclotomic finite multiple zeta values (CFMZVs) and $\mathcal{Z}^{\mathscr{A}(N,\alpha)}$ is their $K_N$-linear span.
In particular, CSMZVs and CFMZVs in the case of $N=1$ have been studied by many researchers  (\cite{Hirose}, \cite{Rosen}, \cite{BTT},etc...).  Indeed, Yasuda proved that CSMZVs span the entire space of CMZVs for $N = 1$.
\begin{thm}[\cite{Yasuda}]\label{thm:Yasuda-1}
We have
\begin{align*}
\widetilde{\mathcal{Z}}^{1}/(\pi i \widetilde{\mathcal{Z}}^1)=  \mathcal{Z}^{\S(1,1)}. 
\end{align*}
\end{thm}
For $N = 1$, if \eqref{eq:KZ-conj} is the well-defined $K_N$-algebra map, then Theorem \ref{thm:Yasuda-1} implies the surjectivity of \eqref{eq:KZ-conj}.

In this paper, we consider a generalization of Theorem \ref{thm:Yasuda-1} for general $N\in\mathbb{Z}_{>0}$.
Our main theorem in this paper, stated as follows, is proved by a reduction to Theorem \ref{thm:Yasuda-1}.

\begin{thm}\label{thm:cYasuda}
For all $N\in \mathbb{Z}_{> 0}$ and $\alpha \in (\mathbb{Z}/N\mathbb{Z})^{\times}$, the equality
\begin{align}
 \widetilde{\mathcal{Z}}^N/ (\pi i\widetilde{\mathcal{Z}}^N) = \mathcal{Z}^{\mathcal{S}(N,\alpha)}\label{eq:Yasu}
\end{align}
holds.
\end{thm}

The case of $N = 2$ was conjectured in \cite[Conjecture 8.6.9]{Zhao}, the case of $N = 3$, $4$ was done in \cite[Conjecture 1.2]{Singer-Zhao}.
The case of general $N$ with arbitrary $\alpha \in \mathbb{Z}/N\mathbb{Z}$ was in \cite[(6.1)]{Tasaka} as a problem.

The reason why we restrict the condition of $\alpha \in (\mathbb{Z}/N\mathbb{Z})^{\times}$ is related to the cyclotomic analogue of the Kaneko-Zagier conjecture \eqref{eq:KZ-conj}, whose $\alpha$ is restricted to $\alpha \in (\mathbb{Z}/N\mathbb{Z})^{\times}$.
For general $N\in\mathbb{Z}_{>0}$ with $\alpha\in(\mathbb{Z}/N\mathbb{Z})^{\times}$, Theorem \ref{thm:cYasuda} implies the surjectivity of \eqref{eq:KZ-conj} if it is the well-defined $K_N$-algebra map.

We expect the equality (\ref{eq:Yasu}) does not hold for $\alpha \notin (\mathbb{Z}/N\mathbb{Z})^{\times}$. For more details, see Remark \ref{rmk:counter-example}.

\section{Cyclotomic multiple zeta values}\label{sec:CMZV}

Let $X$ be the free monoid generated by $\{x_0,x_{\xi}\}_{\xi \in \Gamma_N}$ with the empty word $1$.
Let $\ide{h}_N:=K_N\langle X \rangle$ be the free associative $K_N$-algebra generated by $X$. 
We define two $K_N$-subalgebras $\ide{h}^0_N$ and $\ide{h}_N^1$ of $\ide{h}_N$ by
\[
\ide{h}_N^1:= K_N + \bigoplus_{\substack{w\in x_{\xi}X \\ \xi \in \Gamma_N}} K_N w
\supset
\ide{h}_N^0 := K_N +  \bigoplus_{\substack{w\in x_{\xi}X\\ \xi \in \Gamma_N \\ w\notin Xx_1 }} K_N w.
\]

Let us define the $K_N$-bilinear binary operation $\shuffle$, called the shuffle product, on $\ide{h}_N$ by
\begin{align*}
1\shuffle w =& w \shuffle 1 = w\\
l_1w_1\shuffle l_2w_2 =& l_1(w_1\shuffle l_2w_2) + l_2(l_1w_1 \shuffle w_2 )
\end{align*}
for all $l_1$, $l_2\in \{x_0,x_{\xi}\}_{\xi \in \mu_N}$ and $w$, $w_1$, $w_2\in X$.
The pair $(\ide{h}_N, \shuffle)$ forms a $K_N$-algebra and we denote it by $\ide{h}_{N}^{\shuffle}$.
Then, $\ide{h}_N^1$ and $\ide{h}_N^0$ are closed under $\shuffle$ and become $K_N$-subalgebras of $\ide{h}_N^{\shuffle}$.
We respectively denote them by $\ide{h}_N^{1,\shuffle}$ and $\ide{h}_N^{0,\shuffle}$.

We define a $K_N$-linear map $Z_N : \ide{h}_N^{0,\shuffle} \rightarrow \mathbb{C}$ by $Z_N(1) = 1$ and 
\[
Z_N(x_{\xi_1}x_0^{k_1-1}\cdots x_{\xi_r}x_0^{k_r-1}) = \cmzv{\xi_1\overline{\xi_2}}{\xi_r\overline{\xi_{r+1}}}{k_1}{k_r},
\]
where $(k_1,\ldots,k_r)\in\mathbb{Z}_{>0}^r$ and $(\xi_1,\ldots,\xi_r)\in\Gamma_N^r$ with $(k_r,\xi_r)\neq (1,1)$ and $\xi_{r+1}=1$.
Since the CMZVs have an iterated integral expressions, $Z_N$ is a $K_N$-algebra homomorphism, i.e., for $w_1$ and $w_2\in \ide{h}_N^0$, we have
\begin{equation}
\begin{split}
Z_N(w_1 \shuffle w_2) &= Z_N(w_1)Z_N(w_2).
\end{split}
\label{eq:fds}
\end{equation}
In \cite{Arakawa-Kaneko}, it is shown that
the $K_N$-algebra homomorphim $Z_N : \ide{h}_N^{0,\shuffle}\rightarrow \mathbb{C}$ can be extended uniquely to $Z^{\shuffle}_N : \ide{h}_N^{1,\shuffle}\rightarrow \mathbb{C}$ with $Z^{\shuffle}_N (x_1)= 0$ and $Z_N^{\shuffle} \mid_{\ide{h}_N^0} = Z_N$.
For $\mathbf{k} := (k_1,\ldots,k_r) \in \mathbb{Z}_{>0}^r$ and $\boldsymbol{\xi} := (\xi_1,\ldots,\xi_r)\in\Gamma_N^r$,
we define the \textbf{$\shuffle$-regularized cyclotomic multiple zeta values} by
\[
\zeta^{\shuffle}\left(\begin{matrix}\boldsymbol{\xi} \\ \mathbf{k} \end{matrix}\right) := Z^{\shuffle}_N (x_{\xi_1\cdots \xi_r}x_0^{k_1-1}\cdots x_{\xi_r}x_0^{k_r-1}).
\]
Note that if $(k_r,\xi_r)\neq (1,1)$,  then $\shuffle$-regularized cyclotomic multiple zeta values coincide with
\[
\csmzv{\xi_1}{\xi_r}{k_1}{k_r} = \cmzv{\xi_1}{\xi_r}{k_1}{k_r}.
\]

At the end of this section, we define the $K_N$-linear subspaces $\widetilde{\mathcal{Z}}_w^N$ and $\widetilde{\mathcal{Z}}_{w,d}^N$ of $\widetilde{\mathcal{Z}}^N$ by
\begin{align*}
\widetilde{\mathcal{Z}}_{w}^N =& \Span_{K_N} \left\{ (\pi i)^{s} \zeta\left(\begin{matrix} \boldsymbol{\xi} \\ \mathbf{k} \end{matrix}\right)  \in \widetilde{\mathcal{Z}}_N
\middle| \dep \mathbf{k}, \dep \boldsymbol{\xi} \ge 0, s\ge 0, \text{ and }   \wt \mathbf{k} + s  = w\right\},\\
\widetilde{\mathcal{Z}}_{w,r}^N =& \Span_{K_N} \left\{ ( \pi i)^s \zeta\left(\begin{matrix}\boldsymbol{\xi} \\ \mathbf{k}  \end{matrix}\right)  \in \widetilde{\mathcal{Z}}^N
\middle| 0\le \dep \mathbf{k}, \dep \boldsymbol{\xi} \le r \text{ and } \wt \mathbf{k} + s = w\right\}.
\end{align*}

\section{Cyclotomic symmetric multiple zeta values}
In this section, we define the cyclotomic symmetric multiple zeta values.

\begin{df}
\begin{enumerate}
\item
For $(k_1,\ldots,k_r)\in\mathbb{Z}_{>0}^r$ and $(\xi_1,\ldots,\xi_r)\in \Gamma_N^r$, a \textbf{$\shuffle$-regularized cyclotomic symmetric multiple zeta value} ($\shuffle$-CSMZV) is an element of $\widetilde{\mathcal{Z}}^N$ defined by
\begin{align*}
\begin{split}
\cssmzv{\xi_1}{\xi_r}{k_1}{k_r} :=&  \sum_{j = 0}^r  (\xi_{j+1}\cdots \xi_r)^{\alpha} (-1)^{k_{j+1}+\cdots + k_r} \csmzv{\xi_1}{\xi_j}{k_1}{k_j}\csmzv{\overline{\xi_r}}{\overline{\xi_{j+1}}}{k_r}{k_{j+1}}.
\end{split}
\end{align*}
Here, set $\zeta_{\S(N,\alpha)}^{\shuffle} \left( \begin{matrix} \emptyset \\ \emptyset \end{matrix}\right) = 1$.
\item Under the same settings, we define a \textbf{cyclotomic symmetric multiple zeta value} (CSMZV) by
\[
\cSmzv{\xi_1}{\xi_r}{k_1}{k_r} := \cssmzv{\xi_1}{\xi_r}{k_1}{k_r}\mod   \pi i\widetilde{\mathcal{Z}}^N.
\]
\end{enumerate}
\end{df}

\begin{rmk}\label{rmk:CSMZV-S-H}
We can also define \textbf{$*$-regularized cyclotomic symmetric multiple zeta value} ($*$-CSMZV) $\displaystyle \chsmzv{\xi_1}{\xi_r}{k_1}{k_r}$ as an element of $\widetilde{\mathcal{Z}}^N$. In the same way as in \cite[Theorem 4.6]{Singer-Zhao} 
one can show
\[
\chsmzv{\xi_1}{\xi_r}{k_1}{k_r} \equiv \cssmzv{\xi_1}{\xi_r}{k_1}{k_r} \mod  \pi i \widetilde{\mathcal{Z}}_N.
\]
\end{rmk}

Let $\mathcal{Z}^{\S(N,\alpha),\shuffle}$ the $K_N$-linear subspace of $ \widetilde{\mathcal{Z}}^N$ spanned by $\shuffle$-CSMZVs.
We define the $K_N$-linear subspace $\mathcal{Z}_{w}^{\S(N,\alpha),\shuffle}$ of $\mathcal{Z}_{w}^{\S(N,\alpha)}$ by
\begin{align*}
\mathcal{Z}_{w}^{\S(N,\alpha),\shuffle} =& \Span_{K_N} \left\{  \zeta^{\S(N,\alpha),\shuffle}\left(\begin{matrix} \boldsymbol{\xi} \\ \mathbf{k} \end{matrix}\right)  
\middle|  \wt \mathbf{k}  = w\right\}.
\end{align*}
In the same way, we respectively define $K_N$-linear subspace  $\mathcal{Z}^{\S(N,\alpha)}$ and $\mathcal{Z}^{\S(N,\alpha)}_w$ of $ \widetilde{\mathcal{Z}}^N / (\pi i)$ generated by CSMZVs and those with weight $w$.

The ``harmonic product relations" of the $*$-CSMZV show that $\mathcal{Z}^{\S(N,\alpha)}$ is a $K_N$-algebra (\cite[Proposition 5.2]{Tasaka}), i.e.
\begin{align}
\zeta_{\S(N,\alpha)}\left(\begin{matrix} \boldsymbol{\xi}_1 \\ \mathbf{k}_1 \end{matrix} \right)  \zeta_{\S(N,\alpha)}\left(\begin{matrix} \boldsymbol{\xi}_1 \\ \mathbf{k}_1 \end{matrix} \right) \in \mathcal{Z}^{\S(N,\alpha)} \label{eq:harCSMZV}
\end{align}
for $\mathbf{k}_1\in \mathbb{Z}_{>0}^{d_1}$, $ \boldsymbol{\xi}_1 \in \Gamma_N^{d_1}$, $\mathbf{k}_2\in \mathbb{Z}_{>0}^{d_2}$ and $ \boldsymbol{\xi}_2 \in \Gamma_N^{d_2}$ ($d_1$, $d_2\in\mathbb{Z}_{>0}$).

\section{A proof of Theorem \ref{thm:cYasuda}} 

In this section, we prove Theorem \ref{thm:cYasuda} by reducing to the case of $N = 1$ (Theorem \ref{thm:Yasuda-1}).
The following theorem is a refinement of Theorem \ref{thm:cYasuda}.
\begin{thm}\label{thm:cYasuda-w}
Let $N\in\mathbb{Z}_{>0}$ and $\alpha\in(\mathbb{Z}/N\mathbb{Z})^{\times}$.
For $w\in\mathbb{Z}_{>0}$, we have
\begin{align}
 \widetilde{\mathcal{Z}}_{w}^N / \pi i\widetilde{\mathcal{Z}}_{w-1}^N = \mathcal{Z}_w^{\S(N,\alpha)}. \label{eq:cYasu-w}
\end{align}
\end{thm}
It suffices to prove Theorem \ref{thm:cYasuda-w}.

For $r$, $w\in\mathbb{Z}_{\ge 0}$, we define the $K_N$-linear subspace of $\widetilde{\mathcal{Z}}^N_{w,r}$
\begin{align*}
\mathcal{PD}^{N}_{w,r} :=& \widetilde{\mathcal{Z}}_{w,r-1}^{N} + \sum_{\substack{w_1+w_2 = w \\ w_1,w_2>0}}\sum_{\substack{r_1+r_2\le r \\ r_1,r_2>0}} \widetilde{\mathcal{Z}}_{w_1,r_1}^{N}\cdot \widetilde{\mathcal{Z}}_{w_2,r_2}^{N}.\\
\end{align*}

\begin{lem}[{\cite[Theorem 1.3, Corollary 1.4]{Panzar}}]\label{lem:parity}
For $\mathbf{k} \in \mathbb{Z}_{>0}^r$ with $\wt \mathbf{k} = w$ and $\boldsymbol{\xi} \in \Gamma_N^r$, we have
\[
\zeta^{\shuffle} \left(\begin{matrix} \boldsymbol{\xi} \\ \mathbf{k} \end{matrix} \right)
- (-1)^{w - r}
\zeta^{\shuffle} \left(\begin{matrix} \overline{\boldsymbol{\xi}} \\  \mathbf{k} \end{matrix} \right) \in \mathcal{PD}^{N}_{w,r}. 
\]
Here, let $\overline{\boldsymbol{\xi}}$ be the complex conjugation for all components of $\boldsymbol{\xi}$.
\end{lem}

\begin{proof}
The case of admissible indices is proved in \cite[Theorem 1.3, Corollary 1.4]{Panzar}.
To consider the remaining case, we consider the $\shuffle$-regularization formula (\cite[Section 2.2]{Arakawa-Kaneko}), i.e. 
\[
Z_{N}^{\shuffle}(wx_{\eta}x_1^l) = (-1)^l Z_N^{\shuffle}((w\shuffle x_1^l)x_{\eta})
\]
for $w\in \cup_{\xi \in \Gamma_N} x_{\xi}X \cup\{1\} $, $\eta\in \{0\}\cup \Gamma_N\setminus \{1\}$ and a positive integer $l\in\mathbb{Z}_{> 0}$. This means, for $\mathbf{k}\in\mathbb{Z}_{>0}^r$ and $\boldsymbol{\xi}\in\Gamma_N^r$, there exists admissible indices $(\mathbf{k}_j, \boldsymbol{\xi}_j) \in \mathbb{Z}_{>0}^r \times \Gamma_N^r$ with $\wt \mathbf{k}_j = \wt \mathbf{k}$ and $c_j \in \mathbb{Z}$ such that
\[
\zeta^{\shuffle} \left(\begin{matrix} \boldsymbol{\xi} \\ \mathbf{k} \end{matrix} \right) = \sum_{j} c_j \zeta \left(\begin{matrix} \boldsymbol{\xi}_j \\ \mathbf{k}_j \end{matrix} \right).
\]
Therefore, our claim is reduced to the admissible case.
\end{proof}

\begin{lem}\label{lem:antipode}
For $\mathbf{k}:= (k_1,\ldots,k_r) \in \mathbb{Z}_{>0}^r$ with  $\wt \mathbf{k} = w$ and $\boldsymbol{\xi} := (\xi_1,\ldots,\xi_r) \in \Gamma_N^r$,
\[
\zeta^{\shuffle}\left(\begin{matrix} \boldsymbol{\xi} \\ \mathbf{k} \end{matrix} \right) + (-1)^r\zeta^{\shuffle} \left(\begin{matrix} \overset{\leftarrow}{\boldsymbol{\xi}} \\ \overset{\leftarrow}{\mathbf{k}} \end{matrix} \right)   \in \mathcal{PD}^{N}_{w,r}.
\]
Here, we set $\overset{\leftarrow}{\mathbf{k}} := (k_r,\ldots,k_1)$ and $\overset{\leftarrow}{\boldsymbol{\xi}} := (\xi_r,\ldots,\xi_1)$.
\end{lem}

\begin{proof}
For  $(k_1.\ldots.k_r)\in\mathbb{Z}_{>0}^r$ and $(\xi_1,\ldots,\xi_r)\in \Gamma_N^r$, we have the following relation (\cite[Theorem 4.2]{Hoffman-Ihara})
\[
\sum_{j = 0}^r (-1)^{r -j} \chmzv{\xi_1}{\xi_j}{k_1}{k_j} \zeta^{*,\star}\left(\begin{matrix}\xi_r\hspace{-0.3cm}&,\ldots,&\hspace{-0.3cm} \xi_{j+1} \\  k_r \hspace{-0.3cm}&,\ldots,&\hspace{-0.3cm}  k_{j+1}   \end{matrix}\right) = 0.
\]
Here, $\zeta^*$ (resp. $\zeta^{*,\star}$) is the harmonic regularized multiple zeta (resp. zeta star) value (see \cite[Section 1.2]{Arakawa-Kaneko}). Since
\[
\zeta^{*,\star}\left(\begin{matrix}  \boldsymbol{\xi} \\ \mathbf{k}  \end{matrix}\right) \equiv 
\zeta^{*}\left(\begin{matrix}  \boldsymbol{\xi} \\ \mathbf{k}  \end{matrix}\right) \equiv
\zeta^{\shuffle}\left(\begin{matrix}  \boldsymbol{\xi} \\ \mathbf{k}  \end{matrix}\right)
\]
holds for $\mathbf{k}\in\mathbb{Z}_{>0}^r$ with $\wt \mathbf{k} = w$ and $\boldsymbol{\xi} \in \Gamma_N^r$,
we get our conclusion.

\end{proof}

\begin{lem}\label{lem:cYasuda-pd}
Theorem \ref{thm:cYasuda-w} holds under the assumption that $\widetilde{\mathcal{Z}}_{w,r}^N/(\pi i\widetilde{\mathcal{Z}}_{w-1,r}^N + \mathcal{PD}_{w,r}^N)$ is linearly spanned by the image of $\shuffle$-CSMZVs.
\end{lem}
\begin{proof}
Let us define the $K_N$-linear subspace $\widetilde{\mathcal{PD}}^{N}_{w,r} $ of $(\widetilde{\mathcal{Z}}_{w,r}^N/\pi i\widetilde{\mathcal{Z}}_{w-1,r}^N)$ as the image of $\mathcal{PD}^{N}_{w,r}$ in $\widetilde{\mathcal{Z}}_{w,r}/\pi i \widetilde{\mathcal{Z}}_{w-1,r}$.
Then we can check 
\[
\widetilde{\mathcal{Z}}_{w,r}^N/(\pi i\widetilde{\mathcal{Z}}_{w-1,r}^N + \mathcal{PD}_{w,r}^N)
\cong
(\widetilde{\mathcal{Z}}_{w,r}^N/\pi i\widetilde{\mathcal{Z}}_{w-1,r}^N)/ \widetilde{\mathcal{PD}}^{N}_{w,r}.
\]
Therefore, our assumption can be rephrased as ``$(\widetilde{\mathcal{Z}}_{w,r}^N/\pi i\widetilde{\mathcal{Z}}_{w-1,r}^N)/ \widetilde{\mathcal{PD}}^{N}_{w,r}$ is linearly spanned by the image of CSMZVs."
We prove $\mathcal{Z}_{w}^{\S(N,\alpha)}$ is linearly spanned by its image of CSMZVs by the double induction on $w$ and $r$.

For $\mathbf{k}\in \mathbb{Z}_{>0}^r$ and $\boldsymbol{\xi}\in\Gamma_N^r$, there exists $s\in \mathcal{Z}^{\S(N,\alpha)}_w$ such that
\[
\zeta\left(\begin{matrix} \boldsymbol{\xi} \\ \mathbf{k} \end{matrix} \right) - s \equiv 
\sum_{\mathbf{k}', \boldsymbol{\xi}'} c_{\mathbf{k}', \boldsymbol{\xi}'}    \zeta\left(\begin{matrix}\boldsymbol{\xi}' \\  \mathbf{k}'  \end{matrix} \right)
+
\sum_{\substack{\mathbf{k}'', \mathbf{k}''' \\ \boldsymbol{\xi}'', \boldsymbol{\xi}'''}}  c_{\substack{\mathbf{k}'', \mathbf{k}''' \\ \boldsymbol{\xi}'', \boldsymbol{\xi}'''}}\zeta\left(\begin{matrix}\boldsymbol{\xi}''  \\ \mathbf{k}'' \end{matrix} \right) \times \zeta\left(\begin{matrix}\boldsymbol{\xi}''' \\ \mathbf{k}''' \end{matrix} \right) 
\]
modulo $\pi i \widetilde{\mathcal{Z}}_{w-1,r}^N$,
where $c_{\bullet}$ in $K_N$, 
the first summation of the above runs to $\wt \mathbf{k}' = w$ and $0\le \dep \mathbf{k}' = \dep \boldsymbol{\xi}' \le r-1$
and the second summation of the above runs to $\wt \mathbf{k}'' + \wt \mathbf{k}''' = w$, $\dep \boldsymbol{\xi}'' = \dep \mathbf{k}''$, $\dep \boldsymbol{\xi}''' = \dep \mathbf{k}'''$ and $0\le \dep \mathbf{k}'' + \dep \mathbf{k}''' \le r$,
i.e. the first summation is the image of $\widetilde{\mathcal{Z}}_{w,r-1}^{N}$ in $\widetilde{\mathcal{Z}}_{w,r}^N/\pi i\widetilde{\mathcal{Z}}_{w-1,r}^N$ and the second one of the above is also the image of $\sum_{\substack{w_1+w_2 = w \\ w_1,w_2>0}}\sum_{\substack{r_1+r_2\le r \\ r_1,r_2>0}} \widetilde{\mathcal{Z}}_{w_1,r_1}^{N}\cdot \widetilde{\mathcal{Z}}_{w_2,r_2}^{N} $ in the same space.

Since $\wt \mathbf{k}''$, $\wt \mathbf{k}'''<w$ and CSMZVs are closed under the product of CSMZVs, 
the assumption of the induction implies the right-hand side of the above is a $K_N$-linear combination of CSMZVs. 
\end{proof}

\begin{proof}[Proof of Theorem \ref{thm:cYasuda-w}]

By Lemma \ref{lem:cYasuda-pd} it suffices to show $\widetilde{\mathcal{Z}}_{w,r}^N/(\pi i\widetilde{\mathcal{Z}}_{w-1,r}^N + \mathcal{PD}_{w,r}^N)$ is linearly spanned by the image of $\shuffle$-CSMZVs.

Let $(k_1,\ldots,k_r)\in\mathbb{Z}_{>0}^r$ and $(\xi_1,\ldots,\xi_r)\in \Gamma_N^r$.

Our proof is splitted into the following two steps:
\begin{enumerate}
\item If $\xi_1\cdots \xi_r \neq 1$, then  $\displaystyle \csmzv{\xi_1}{\xi_r}{k_1}{k_r} \mod \pi i\widetilde{\mathcal{Z}}_{w-1,r}^N + \mathcal{PD}_{w,r}^N$  is a $K_N$-linear combination of the image of  $\shuffle$-CSMZVs.
\item The general cases are reduced to the first step and Theorem \ref{thm:Yasuda-1}.
\end{enumerate}

First, we prove that  $\displaystyle \csmzv{\xi_1}{\xi_r}{k_1}{k_r} \mod \pi i \widetilde{\mathcal{Z}}_{w-1}^N + \mathcal{PD}_{w,r}^N$ with $\xi_1\cdots \xi_r \neq 1$ is a $K_N$-linear combination of the image of $\shuffle$-CSMZVs.
We have 

\begin{align*}
\cSmzv{\xi_1}{\xi_r}{k_1}{k_r} \equiv& \csmzv{\xi_1}{\xi_r}{k_1}{k_r} + (-1)^{k_1+\cdots + k_r} (\xi_1\cdots \xi_r)^{\alpha} \csmzv{\overline{\xi_r}}{\overline{\xi_1}}{k_r}{k_1}\\
\equiv & \csmzv{\xi_1}{\xi_r}{k_1}{k_r} + (-1)^{r}(\xi_1\cdots \xi_r)^{\alpha} \csmzv{\xi_r}{\xi_1}{k_r}{k_1}\\
\equiv & \csmzv{\xi_1}{\xi_r}{k_1}{k_r} - (\xi_1\cdots \xi_r)^{\alpha}\csmzv{\xi_1}{\xi_r}{k_1}{k_r}\\
\equiv& (1 -(\xi_1\cdots \xi_r)^{\alpha}) \csmzv{\xi_1}{\xi_r}{k_1}{k_r} \mod \pi i \widetilde{\mathcal{Z}}_{w-1}^N + \mathcal{PD}_{w,r}^N.
\end{align*}
Now we can use Lemma \ref{lem:parity} in the second equality and Lemma \ref{lem:antipode} in the third equality. 
Since $\xi_1\cdots \xi_r \neq 1$, $1 - (\xi_1 \cdots \xi_r)^{\alpha} \neq 0$ by  $\alpha\in (\mathbb{Z}/N\mathbb{Z})^{\times}$ and our 1-st step is completed.

Second, we consider the general case.
It suffices to exclude the case $\xi_1 = \cdots = \xi_r = 1$, as this case follows directly from Theorem \ref{thm:Yasuda-1}.
Since
\begin{align}
\csmzv{\xi_1}{\xi_r}{k_1}{k_r} = Z_N^{\shuffle}(x_{\xi_1\cdots \xi_r}x_0^{k_1-1}\cdots x_{\xi_r}x_0^{k_r-1}) \label{eq:1-st}
\end{align}
holds, $\displaystyle \csmzv{\xi_1}{\xi_r}{k_1}{k_r}$ with $\xi_1\cdots \xi_r  = 1$ is reduced to $Z_N^{\shuffle}(x_{\eta_1}x_0^{k_1-1}\cdots x_{\eta_r}x_0^{k_r-1})$ for $\eta_1,\ldots,\eta_r\in \Gamma_N$ with $\eta_1 = 1$.

We prove, by induction on $l\in\mathbb{Z}_{\ge 0},$ that $Z^{\shuffle}_N(x_{i_1} \cdots x_{i_l} x_{\xi_1} x_0^{k_1-1} \cdots x_{\xi_r} x_0^{k_r-1}) \mod \pi i \widetilde{\mathcal{Z}}_{w-1}^N + \mathcal{PD}_{w,r}^N$ ($i_2, \ldots, i_l \in \{0, 1\}$, $x_{i_1} = x_1$, and $\xi_1 \neq 1$) can be expressed as a $K_N$-linear combination of
\[
Z_N^{\shuffle}(x_{\xi_1'} x_0^{k_1'-1} \cdots x_{\xi_r'} x_0^{k_{r'}'-1}) \mod \pi i \widetilde{\mathcal{Z}}_{w-1}^N + \mathcal{PD}_{w,r}^N,
\]
where $k_1', \ldots, k_{r'}' \in \mathbb{Z}_{>0}$ and $\xi_1', \ldots, \xi{r'}' \in \Gamma_N$ with $\xi_1' \neq 1$.
Let us assume that the claim holds for all $l' < l$.

Since there exists $w_{j}\in K_N\langle X\rangle$ ($ j = 0,\ldots,l-1$) and $w\in K_N\langle X\rangle$ such that
\begin{align*}
&x_1x_{i_1}\cdots x_{i_l} x_{\xi_1}x_0^{k_1-1}\cdots x_{\xi_r}x_0^{k_r-1}\\
=&x_1x_{i_1}\cdots x_{i_l} \shuffle x_{\xi_1}x_0^{k_1-1}\cdots x_{\xi_r}x_0^{k_r-1} - \sum_{0\le j < l} x_1x_{i_1}\cdots x_{i_{j}}x_{\xi_1}w_{j} - x_{\xi_1}w
\end{align*}
holds, then we have
\begin{align*}
&Z^{\shuffle}_N(x_1x_{i_1}\cdots x_{i_l} x_{\xi_1}x_0^{k_1-1}\cdots x_{\xi_r}x_0^{k_r-1})\\
=&Z^{\shuffle}_N(x_1x_{i_1}\cdots x_{i_l}) Z^{\shuffle}_N( x_{\xi_1}x_0^{k_1-1}\cdots x_{\xi_r}x_0^{k_r-1})  - \sum_{0\le j < l} Z^{\shuffle}_N(x_1x_{i_1}\cdots x_{i_{j}}x_{\xi_1}w_{j} )- Z^{\shuffle}_N( x_{\xi_1}w)\\
\equiv 
 &- \sum_{0\le j < l} Z^{\shuffle}_N(x_1x_{i_1}\cdots x_{i_{j}}x_{\xi_1}w_{j} )- Z^{\shuffle}_N( x_{\xi_1}w).
\end{align*}
The last congruence is taken modulo $\pi i \widetilde{\mathcal{Z}}_{w-1}^N + \mathcal{PD}_{w,r}^N$.
Since $1\le j \le l-1$, the assumption of the induction implies $\displaystyle \sum_{0\le j < l} Z^{\shuffle}_N(x_1x_{i_1}\cdots x_{i_{j}}x_{\xi_1}w_{j} ) \mod \pi i \widetilde{\mathcal{Z}}_{w-1}^N + \mathcal{PD}_{w,r}^N$ can be written as a $ K_N$-linear combination of $Z_N^{\shuffle}(x_{\xi_1'}x_0^{k_1'-1}\cdots x_{\xi'_r}x_0^{k_{r'}'-1}) \mod \pi i \widetilde{\mathcal{Z}}_{w-1}^N + \mathcal{PD}_{w,r}^N$ with $\xi_1' \neq 1$.
Therefore, we get the conclusion.

\end{proof}

\begin{cor}\label{cor:cYasuda-sh}
We have
\[
\widetilde{\mathcal{Z}}^N =  \mathcal{Z}^{\S(N,\alpha),\shuffle}[\pi i].
\]
More precisely, we have
\[
\widetilde{\mathcal{Z}}^N_w =  \sum_{j = 0}^w (\pi i)^j \mathcal{Z}^{\S(N,\alpha),\shuffle}_{w-j}.
\]
 for $w\in\mathbb{Z}_{\ge 0}$.
\end{cor}

\begin{proof}
We prove this claim by induction on $w$. Assume our induction hypothesis holds less than $w' < w$ for fixed $w\in\mathbb{Z}_{>0}$.
Let $z \in \widetilde{Z}_{w}^N$. By Theorem \ref{thm:cYasuda-w} and the definition of CSMZV, there exists $s_w \in \mathcal{Z}^{\S(N,\alpha),\shuffle}_w$ and $z_{w-1} \in \widetilde{Z}_{w-1}^N$ such that
\[
z = s_w + \pi i \cdot  z_{w-1}
\]
By the assumption of the induction, there exists $s_j \in \mathcal{Z}^{\S(N,\alpha),\shuffle}_{w-j}$ $(0\le j \le w-1)$ such that
$z_{w-1} = \sum_{j = 0}^{w-1} (\pi i)^{w-j-1}s_j$. Therefore, our claim is completed.
\end{proof}

\begin{rmk}
In \cite[Lemma 3.10]{Furusho-Komiyama}, Lemma 4.1 is proved using only the double shuffle relations. Therefore, the above proof can also work for analogs of multiple zeta values that satisfy the double shuffle relations. For instance, the cyclotomic motivic multiple zeta values and the formal cyclotomic multiple zeta values.
\end{rmk}

\begin{rmk}\label{rmk:counter-example}
The equality \eqref{eq:cYasu-w} does not hold in the case of $\alpha = 0$. We can verify $\zeta^{\S(N,0)}\left( \begin{matrix} \xi \\ 1 \end{matrix}\right) = 0$ for $N\in\mathbb{Z}_{>0}$ and $\xi\in \Gamma_N$. This implies the $K_N$-linear space generated by CSMZV with weight $1$ is equal to $0$.
However, according to the Baker's theorem (\cite[Theorem 2.1]{Baker}), since a pair $\zeta \left( \begin{matrix} \xi \\ 1 \end{matrix} \right) = - \log (1-\xi)$ and $\pi i$ is $\mathbb{Q}$-linear independent for $N \ge 2$ and $N\neq 6$, a triple $1$, $-\log(1-\xi)$ and $\pi i$ is $\overline{\mathbb{Q}}$-linear independent. This implies the $K_N$-linear space generated by CMZV with weight $1$  modulo $\pi i$ is a non-trivial space and provides a counter-example to Theorem \ref{thm:cYasuda-w} for $\alpha = 0$.
Therefore, since we expect that all the spaces involved are graded
\[
\widetilde{\mathcal{Z}}^{N} = \bigoplus_{w\ge 0} \widetilde{\mathcal{Z}}_w^{N},
\]
this would provide a counter-example to \eqref{eq:Yasu}.

\end{rmk}

\section{Refined symmetric multiple zeta values and generalization of Theorem \ref{thm:cYasuda}}\label{sec:RSCMZV}

In this section, we describe a refinement of Theorem \ref{thm:cYasuda}.
Tasaka introduced refined cyclotomic symmetric multiple zeta values (RCSMZVs) in \cite{Tasaka}.
The original case was studied in \cite{Jarossay}, while independently, the case $N = 1$ is considered in \cite{Hirose}, and Tasaka extended it to general $N > 0$.

For $(k_1,\ldots,k_r)\in\mathbb{Z}_{>0}^r$ and $(\xi_1,\ldots,\xi_r)\in \Gamma_N^r$,
Tasaka (\cite{Tasaka}) defines the refined cyclotomic symmetric multiple zeta values $\crsmzv{\xi_1}{\xi_r}{k_1}{k_r}$, which is the lift of our CSMZVs, as the element of $ \widetilde{\mathcal{Z}}^N$.
Then we have
\begin{align}
\zeta_{\RS(N,\alpha)}\left(\begin{matrix} \boldsymbol{\xi} \\ \mathbf{k} \end{matrix}\right) \equiv &
\zeta_{\S(N,\alpha)}^{\shuffle}\left(\begin{matrix} \boldsymbol{\xi} \\ \mathbf{k} \end{matrix}\right) \mod  \pi i\widetilde{\mathcal{Z}}^N.
\label{eq:RS=S}
\end{align}
for each $\mathbf{k}\in\mathbb{Z}_{>0}^r$ and $\boldsymbol{\xi}\in\Gamma_N^r$ (see \cite[Definition 3.1]{Tasaka} and Remark \ref{rmk:CSMZV-S-H}).

Let $\mathcal{Z}^{\RS(N,\alpha)}$ be a $K_N$-linear subspace of $ \widetilde{\mathcal{Z}}^N$ generated by RCSMZVs.
We define $K_N$-linear subspace $\mathcal{Z}_w^{\RS(N,\alpha)}$ of $\widetilde{\mathcal{Z}}^N$ by
\begin{align*}
\mathcal{Z}_{w}^{\RS(N,\alpha)} =& \Span_{K_N} \left\{  \zeta_{\RS(N,\alpha)}\left(\begin{matrix} \boldsymbol{\xi} \\ \mathbf{k} \end{matrix}\right)  
\middle|  \wt \mathbf{k}  = w\right\}.
\end{align*}
The ``harmonic product relations" of RCSMZV show that $\mathcal{Z}^{\RS(N,\alpha)}$ is a $K_N$-algebra (\cite[Proposition 5.2]{Tasaka}), i.e.
\begin{align}
\zeta_{\RS(N,\alpha)}\left(\begin{matrix} \boldsymbol{\xi}_1 \\ \mathbf{k}_1 \end{matrix} \right)  \zeta_{\RS(N,\alpha)}\left(\begin{matrix} \boldsymbol{\xi}_2 \\ \mathbf{k}_2 \end{matrix} \right) \in \mathcal{Z}^{\RS(N,\alpha)} \label{eq:RSds}
\end{align}
for $\mathbf{k}_1\in \mathbb{Z}_{>0}^{d_1}$, $ \boldsymbol{\xi}_1 \in \Gamma_N^{d_1}$, $\mathbf{k}_2\in \mathbb{Z}_{>0}^{d_2}$ and $ \boldsymbol{\xi}_2 \in \Gamma_N^{d_2}$ ($d_1$, $d_2\in\mathbb{Z}_{>0}$).

We can generalize Theorem \ref{thm:cYasuda-w} in the following theorem, which settles a special case of a problem proposed for arbitrary $\alpha \in \mathbb{Z}/N\mathbb{Z}$ in \cite[(6.1)]{Tasaka}.
\begin{thm}\label{thm:cYasuda-r}
Let $N\in\mathbb{Z}_{>0}$ and $\alpha\in(\mathbb{Z}/N\mathbb{Z})^{\times}$.
For $w\in\mathbb{Z}_{>0}$, we have
\begin{align}
\mathcal{Z}_w^{\RS(N,\alpha)} =  \widetilde{\mathcal{Z}}_{w}^N. \label{eq:cYasuda-refinement}
\end{align}
\end{thm}

\begin{proof}
Since $\mathcal{Z}^{\RS(N,\alpha)}_w\subset  \widetilde{\mathcal{Z}}_w^N$ holds by the definition of CRSMZVs, we prove the opposite inclusion.

We prove this by induction on $w$. The case of $w = 0$ is trivial. Assume $\mathcal{Z}^{\RS(N,\alpha)}_{w-1} \supset \widetilde{\mathcal{Z}}_{w-1}^N$.
Let $c \in  \widetilde{\mathcal{Z}}_w^N$. By Theorem \ref{thm:cYasuda-w} and \eqref{eq:RS=S}, there exists $y\in \mathcal{Z}^{\RS(N,\alpha)}_w$ such that $c - y \in  \pi i \widetilde{\mathcal{Z}}_{w-1}^{N}$. Therefore,
\[
\frac{c - y}{\pi i} \in  \mathcal{Z}_{w-1}^{N} = \mathcal{Z}_{w-1}^{\RS(N,\alpha)}.
\]
By the special case of \eqref{eq:RSds} and $\zeta^{\RS(N,\alpha)}\left( \begin{matrix} 1 \\ 1 \end{matrix} \right) = - \pi i$ (\cite[Below Definition 3.1]{Tasaka}), we have
\[
-\pi i \zeta^{\RS(N,\alpha)}\left(\begin{matrix}  \boldsymbol{\xi} \\ \mathbf{k}\end{matrix}\right) 
=
  \zeta^{\RS(N,\alpha)}\left(\begin{matrix} 1 \\ 1 \end{matrix}\right) \zeta^{\RS(N,\alpha)}\left(\begin{matrix}  \boldsymbol{\xi} \\ \mathbf{k}\end{matrix}\right) 
\in \mathcal{Z}_{w}^{\RS(N,\alpha)},
\]
for $\mathbf{k}\in \mathbb{Z}_{>0}^r$ with $\wt \mathbf{k} = w-1$ and $\boldsymbol{\xi}\in \Gamma_N^r$. Therefore we have $\pi i \mathcal{Z}_{w-1}^{\RS(N,\alpha)} \subset \mathcal{Z}_{w}^{\RS(N,\alpha)}$. Thus $c = (c-y) + y\in \mathcal{Z}_{w}^{\RS(N,\alpha)}$ which complete the proof.
\end{proof}

Theorem \ref{thm:cYasuda-w} recovers by taking the quatient $\pi i \widetilde{Z}_{w-1}^N$ to \eqref{eq:cYasuda-refinement}.

\begin{center}
Acknowledgment
\end{center}

The author also would like to thank Professor Minoru Hirose, Professor Hidekazu Furusho, and Professor Henrik Bachmann for carefully reviewing the detailed structure of the paper.
This work was financially supported by JST SPRING, Grant Number JPMJSP2125. The author would like to thank the “Interdisciplinary Frontier Next-Generation Researcher Program of the Tokai Higher Education and Research System."

\bibliographystyle{amsplain}
\bibliography{CSMZVspan}

\end{document}